\documentclass[12pt]{article}
\usepackage{mathtools,amsfonts}
\usepackage{mathrsfs,hyperref}
\usepackage{amssymb,amsthm,tikz,stmaryrd}
\usepackage{fancyhdr,sectsty}
\usetikzlibrary{arrows}

\def\name{Nan Shi}
\hypersetup{
  colorlinks = true,
  urlcolor = black,
  pdfauthor = {\name},
  pdfkeywords = {},
  pdftitle = {},
  pdfsubject = {Mathematics},
  pdfpagemode = UseNone
}

\title{Almost-Gelfand Property of Symmetric Pairs}
\author{\name}
\date{}

\sectionfont{\sffamily\bfseries\LARGE}
\subsectionfont{\sffamily\mdseries\itshape\Large}

\newtheorem{theorem}[subsubsection]{Theorem}
\newtheorem{lemma}[subsubsection]{Lemma}
\newtheorem{corollary}[subsubsection]{Corollary}
\newtheorem{proposition}[subsubsection]{Proposition}

\theoremstyle{definition}
\newtheorem{definition}[subsubsection]{Definition}

\theoremstyle{plain}

\newtheorem{notation}[subsubsection]{Notation}

\numberwithin{equation}{section}

\begin{document}

\maketitle
\begin{abstract}
In this paper, we show almost-Gelfand property of connected symmetric pairs $(G, H)$ over finite fields of large characteristics. We will show almost-$\sigma$-invariant property of double coset $H \backslash G /H$ where $\sigma$ is the associated anti-involution, and $\epsilon$-version of Gelfand's trick to make use of the fixed points of anti-involution.
\end{abstract}

\thispagestyle{empty}

\tableofcontents

\newpage
\setcounter{page}{1}
\section{Introduction}
\subsection{Symmetric pairs and Gelfand pairs}
\begin{definition}
A \textbf{symmetric pair} is a triple $(G, H, \theta)$ where $H \subset G$ is a subgroup, and $\theta$ is an involution (i.e. $\theta^2 = 1$) of $G$ such that $H = G^\theta$. For a symmetric pair $(G, H, \theta)$ we define the associated anti-involution $\sigma: G \rightarrow G$ by $\sigma(g) \coloneqq \theta(g^{-1})$.
\end{definition}

Symmetric pairs are important objects in representation theory. Questions about the spectrum decomposition of the symmetric space $X \coloneqq G/H$ have been studied in locally compact case, in relation with number theory via Langlands program.

\begin{definition}
A \textbf{Gelfand pair} is a finite group $G$ and their subgroup $H$, such that for any irreducible representation $\pi$ of $G$ $\dim\mathrm{Hom}_H(\pi, \mathbb{C}) \leq 1$.
\end{definition}

The notion of Gelfand property can be generalized to other types of groups. For finite groups and compact topological groups, there are equivalent definitions for Gelfand pairs:

\begin{lemma}
When $G$ is a finite group, the following are equivalent:
\begin{itemize}
\item $(G, H)$ is a Gelfand pair.
\item The permutation representation $\mathbb{C}[G/H]$ of $G$ is multiplicity-free.
\item The algebra of $(H, H)$-invariant functions on $G$ with multiplication defined by convolution is commutative.

\end{itemize}
\end{lemma}

\begin{proof}
See \cite{Angeli} Section 1, and Theorem 2.3, as a simple application of Frobenius reciprocity.
\end{proof}

\begin{lemma}
When $G$ is a compact topological group, the following are equivalent:
\begin{itemize}
\item $(G, H)$ is a Gelfand pair.
\item The representation $\mathrm{L}^2(G/H)$ of $G$ is multiplicity-free.
\item The algebra of $(H, H)$-invariant compactly supported measures on $G$ with multiplication defined by convolution is commutative.

\end{itemize}
\end{lemma}

\begin{proof}
See \cite{Yakimova} Definition 1.
\end{proof}

A simple criterion for the condition of the last equivalent definition was introduced by Gelfand:

\begin{proposition}[Gelfand's trick]\label{Gelfand_trick}
If there is an anti-involution $\sigma$ acting trivially on the algebra of $(H, H)$-invariant functions on $G$, then the algebra is commutative.
\end{proposition}

\begin{proof}
See \cite{Ryan} Proposition 11. A brief idea of the proof is that $\sigma$-action on the double coset $H \backslash G /H$ reverses the order of multiplication
\[
Hg_1H \cdot Hg_2H = \sigma(Hg_1H \cdot Hg_2H) = \sigma(Hg_2H)\sigma(Hg_1H) = Hg_2H \cdot Hg_1H.
\]

In particular, this shows that $(G, H)$ is a Gelfand pair.
\end{proof}

For reductive groups over a local field, we have similar notions of Gelfand pairs:

\begin{lemma}
When $G$ is a reductive groups over a local field and $H$ is a closed subgroup, there are three non-equivalent notions of Gelfand pairs: (in fact $GP1 \Rightarrow GP2 \Rightarrow GP3$)
\begin{itemize}
\item (GP1) For any irreducible admissible representation $\pi$ of $G$ $\dim\mathrm{Hom}_H(\pi, \mathbb{C}) \leq 1$.
\item (GP2) For any irreducible admissible representation $\pi$ of $G$, $\tilde\pi$ denotes the smooth dual, we have $\dim\mathrm{Hom}_H(\pi, \mathbb{C})\cdot\dim\mathrm{Hom}_H(\tilde\pi, \mathbb{C}) \leq 1$.
\item (GP3) For any unitary representation $\pi$ of $G$ on Hilbert spaces $\dim\mathrm{Hom}_H(\pi, \mathbb{C}) \leq 1$.
\end{itemize}
\end{lemma}

It is classically known that if the field $F$ is Archimedean, $G$ is a connected Lie group and $H$ is compact, then $(G,H)$ is a Gelfand pair. The Gelfand property is often satisfied by symmetric pairs, for example $\big(\mathrm{GL}_{n + m}(\mathbb{C}), \mathrm{GL}_{m + n}(\mathbb{C})\big)$, $\big(\mathrm{GL}_n(\mathbb{C}), \mathrm{O}_n(\mathbb{C})\big)$, $\big(\mathrm{O}_{n+m}(\mathbb{C}), \mathrm{O}_n(\mathbb{C}) \times \mathrm{O}_m(\mathbb{C})\big)$, see \cite{Jacquet}, \cite{Aizenbud3}, \cite{Sun}, \cite{Aizenbud}, \cite{Aizenbud2}. These papers extend Gelfand's trick and examine the anti-involution on $G$. They reduced the verification of the Gelfand property to the understanding of two-sided $H$-orbits on $G$. The method of Harish-Chandra descent is introduced to reduce Gelfand property of a symmetric pair to its descendants inductively.

\subsection{Main Results}

In this paper, we are interested in symmetric pairs over finite fields, where Gelfand's trick cannot be applied, because there are not anti-involutions acting trivially. Consider the example of $(\mathrm{GL}_2(\mathbb{F}_q), T)$, where $T$ is the maximal torus of diagonal matrices. $|T \backslash \mathrm{GL}_2(\mathbb{F}_q) / T| = q + 4$ while there are only $q + 2$ fixed points of the anti-involution. On the representation side of things, the Steinberg representation appears in $\mathbb{C}[\mathrm{GL}_2(\mathbb{F}_q)/T]$ twice, which shows that $(\mathrm{GL}_2(\mathbb{F}_q), T)$ is not a Gelfand pair.

We will consider symmetric pairs over finite fields as $F$-points of symmetric pair of group schemes.

\begin{definition}
A \textbf{symmetric pair of group schemes} is a triple $(\mathbf{G}, \mathbf{H}, \theta)$ where $\mathbf{H} \subset \mathbf{G}$ are reductive group schemes over $\mathbb{Z}$, and $\theta$ is an involution of $\mathbf{G}$ such that $\mathbf{H} = \mathbf{G}^\theta$. We call a symmetric pair \textbf{connected} if $\mathbf{G}/\mathbf{H}$ is connected. We also define an anti-involution $\sigma: \mathbf{G} \rightarrow \mathbf{G}$ by
\[
\sigma(g) \coloneqq \theta(g^{-1}).
\]
\end{definition}

We will develop a quantitative relation between $\sigma$-invariant property of the algebra of $(H, H)$-invariant functions and the multiplicity-free property of the permutation representation $\mathbb{C}[G/H]$. We use a similar idea of Gelfand's trick analysing fixed points of anti-involution on the algebra of $(H, H)$-invariant functions to conclude that it is almost-$\sigma$-invariant. As a result, we obtain almost-Gelfand property of connected symmetric pairs, i.e. most irreducible components in the decomposition of $\mathbb{C}[G/H]$ are of multiplicity one.

\begin{definition}
Let $\big\{(G_i, H_i)\big\{$ be a family finite groups and its subgroup, and denote $X_i \coloneqq G_i/H_i$. This family is called \textbf{almost-Gelfand} if
\[
\frac{\Big|\big\{\rho \in \mathrm{Irr}(G_i) \mid \dim\mathrm{Hom}_{G_i}(\rho, \mathbb{C}[X_i]) = 1\big\}\Big|}{\Big|\big\{\rho \in \mathrm{Irr}(G_i) \mid \dim\mathrm{Hom}_{G_i}(\rho, \mathbb{C}[X_i]) > 0\big\}\Big|} \xrightarrow{|X_i| \rightarrow \infty} 1.
\]
\end{definition}

\begin{definition}
Let $\{Z_i, \sigma_i\}$ be a family of $\mathbb{C}$-algebras together with their automorphisms. Then this family is called\textbf{ almost-$\sigma$-invariant} if
\[
\frac{\dim\{z \in Z_i \mid \sigma_i(z) = z\}}{\dim Z_i} = \frac{\dim Z_i^{\sigma_i}}{\dim Z_i} \xrightarrow{\dim Z_i \rightarrow \infty} 1
\]
\end{definition}

The main result of this paper is the following theorem:

\begin{theorem}\label{multi_one_thm}
Let $\big(\mathbf{G}, \mathbf{H}, \theta\big)$ be a connected symmetric pair of group schemes and $F = \mathbb{F}_q$ be a finite field of characteristic $p$. Denote $\mathbb{C}[\mathbf{G}(F)/\mathbf{H}(F)]$ the permutation representation of $\mathbf{G}(F)$, and $\mathrm{Irr}\big(\mathbf{G}(F)\big)$ the set of irreducible representations of $\mathbf{G}(F)$. Suppose $\mathbf{H}_x$ (the stabilizer of $x$ with respect to conjugation) is connected for all $x \in \mathbf{G}$ semi-simple. Then there is a prime $p_0$ such that for all characteristics $p > p_0$,
\[
\frac{\Big|\{\rho \in \mathrm{Irr}\big(\mathbf{G}(F)\big) \mid \dim\mathrm{Hom}\big(\rho, \mathbb{C}[\mathbf{G}(F)/\mathbf{H}(F)]\big) = 1\}\Big|}{\Big|\{\rho \in \mathrm{Irr}\big(\mathbf{G}(F)\big) \mid \dim\mathrm{Hom}\big(\rho, \mathbb{C}[\mathbf{G}(F)/\mathbf{H}(F)]\big) > 0\}\Big|} \geq 1 - \frac{2C}{q},
\]
where $C$ is a constant depending on the scheme $\mathbf{G}$ but not on the field $F$.
\end{theorem}

\begin{corollary}\label{asym_thm}
From the formula above, under the same conditions, we have the following asymptotic conclusion:
\[
\lim_{q \rightarrow \infty} \frac{\Big|\{\rho \in \mathrm{Irr}\big(\mathbf{G}(F)\big) \mid \dim\mathrm{Hom}\big(\rho, \mathbb{C}[\mathbf{G}(F)/\mathbf{H}(F)]\big) = 1\}\Big|}{\Big|\{\rho \in \mathrm{Irr}\big(\mathbf{G}(F)\big) \mid \dim\mathrm{Hom}\big(\rho, \mathbb{C}[\mathbf{G}(F)/\mathbf{H}(F)]\big) > 0\}\Big|} = 1.
\]
\end{corollary}

This means that the family $\big(\mathbf{G}(\mathbb{F}_q), \mathbf{H}(\mathbb{F}_q)\big)_{q \rightarrow \infty}$ is almost-Gelfand. The restriction of characteristics $p > p_0$ mainly comes from the construction of exponential map from $\mathfrak{g}$ to $G$ as in the proof of Lemma $\ref{sl_lemma}$. The main theorem comes from the following two theorems, the first one showing that the algebra of $(H, H)$-invariant functions is almost-$\sigma$-invariant and the second one similar to Gelfand's trick in \ref{Gelfand_trick}.

\begin{theorem}\label{dim_thm}
Under the same condition of Theorem \ref{multi_one_thm}, denote $Z \coloneqq \mathbf{H}(F) \backslash \mathbf{G}(F) \slash \mathbf{H}(F)$ the set of double cosets, which can also be viewed as the set of $\mathbf{H}(F) \times \mathbf{H}(F)$-orbits. The anti-involution $\sigma$ on $\mathbf{G}(F)$ extends to a map $\sigma: Z \rightarrow Z$ in the obvious way. Then there is a prime $p_0$ such that for all characteristics $p > p_0$,
\[
|Z^\sigma| \geq (1 - \frac{C}{q})|Z|,
\]
where $C$ is a constant depending on the scheme $\mathbf{G}$ but not on the field $F$.
\end{theorem}

\begin{proof}
This means that the algebra of $(H(\mathbb{F}_q), H(\mathbb{F}_q))$-invariant functions is almost-$\sigma$-invariant. The proof of this theorem will be given in $\S$\ref{ch_proof}. We will use geometric properties of symmetric pairs discussed in $\S$\ref{ch_sym}.
\end{proof}

\begin{theorem}[$\epsilon$-version of Gelfand's trick]\label{e_Gel}
Let $G$ be a finite group and $H$ its subgroup. Denote $Z \coloneqq H \backslash G /H$. If there is an anti-involution $\sigma$ on $G$ which can be extended to a map $\sigma: \mathbb{C}[Z] \rightarrow \mathbb{C}[Z]$, such that $\mathrm{codim} \mathbb{C}[Z]^\sigma \leq \epsilon\dim\mathbb{C}[Z]$, then
\[
\frac{\Big|\{\rho \in \mathrm{Irr}(G) \mid \dim\mathrm{Hom}\big(\rho, \mathbb{C}[G/H]\big) = 1\}\Big|}{\Big|\{\rho \in \mathrm{Irr}(G) \mid \dim\mathrm{Hom}\big(\rho, \mathbb{C}[G/H]\big) > 0\}\Big|} \geq 1 - 4\epsilon,
\]
\end{theorem}

\begin{proof}
The proof of this theorem will be given in $\S$\ref{ch_proof}. Tools for counting multiplicity-one irreducible representations are discussed in $\S$\ref{ch_Gel}.
\end{proof}

To study symmetric pairs over finite fields, we first need to analyse the situation over their algebraic closure where the tools of categorical quotient can be applied. Each point in the categorical quotient is associated with a unique closed orbit. Fixed points of the anti-involution are related to the geometric property of closed $H\times H$-orbits. There are two important results about points in categorical quotient: for most points the closed orbit is the only orbit in it; for the rest, there are only finitely many non-closed orbits in each fibre. With these relations and the help of Hilbert90, we can estimate corresponding orbits over the original finite field.

Since $\mathbb{C}[G/H]$ is a semi-simple algebra, we investigate all possible anti-involutions on semi-simple algebras. The dimension estimation of fixed points of anti-involutions in symmetric pairs reduces to estimation of multiplicity-one irreducible representations in the decomposition of $\mathbb{C}[G/H]$.

\subsection{Structure of the Paper}

In $\S$\ref{ch_sym} we discuss properties of symmetric pairs. We are most interested in the standard two-side action of $H \times H$ on $G$, i.e. $(h_1, h_2) \cdot g \coloneqq h_1gh_2^{-1}$. We could also view this action under symmetrization map as conjugation action of $H$ on $G^\sigma$. Closed orbits come into the picture since they are fixed by $\sigma$. The notion of categorical quotient plays an important role for estimating the number of closed orbits. Each point in the categorical quotient is associated with a unique closed orbit. We will also see that closed orbits in $G^\sigma$ correspond to orbits of semi-simple elements. Lastly, by finding a dense open subset of semi-simple elements, we will conclude that most orbits are closed, and therefore fixed by $\sigma$.

In $\S$\ref{ch_Gel} we use Schur's lemma to analyse fixed points of anti-involution on semi-simple algebras. We started by looking at the easiest case of matrix algebras, and use them to calculate more general cases.

In $\S$\ref{ch_proof} we use all the tools we built up to show our main theorems using the ideas described in the introduction.

\subsection{Acknowledgement}

I would like to thank my advisor, Nir Avni, for suggesting this project, generously offering his time and guidance as I developed this paper, and consistently supporting my mathematical study throughout the years. I would also like to thank Avraham Aizenbud for generously explaining some crucial points related to this topic. Special thanks to Michel Brion for his help with the understanding of symmetric pairs.

\section{Symmetric Pairs}\label{ch_sym}

In this section, many notations and results about symmetric pairs are from \cite{Aizenbud} Section 7.1. Also, many results about $H$-conjugation orbits are from \cite{Kostant} Chapter I and II, where $H = K_\theta$, $\mathfrak{h} = \mathfrak{f}$, $\mathfrak{g}^\sigma = \mathfrak{p}$, $\mathfrak{sl}_2$-triple as $S$-triple.
\subsection{Preliminaries and Notations}
\begin{notation}
\begin{itemize}
    \item We fix a finite field $F = \mathbb{F}_q$ of characteristics $p \neq 2$ and its algebraic closure $\bar F = \mathbb{\bar F}_q$. Most algebraic varieties and algebraic groups we will consider are over the algebraic closure $\bar F$ unless otherwise specified, i.e. $G = \mathbf{G}(\bar F)$, $G(F) = \mathbf{G}(F)$, etc.

    \item For a group $G$ acting on a variety $X$, we denote $X^G$ the fixed points of $X$ by $G$. Also we denote $G_x$ the stabilizer of $x$ in $G$. Denote $\mathcal{U}$ the set of unipotent elements in $G$.

    \item For their corresponding Lie algebra, denote $\mathfrak{g} \coloneqq \mathrm{Lie}G$, $\mathfrak{h} \coloneqq \mathrm{Lie}H$, and $\mathfrak{n}$ nilpotent elements.
\end{itemize}
\end{notation}

\begin{definition}
For a symmetric pair $(G, H, \theta)$, we can define a symmetrization map $s: G \rightarrow G^{\sigma}$ by
\[
s(g) \coloneqq g\sigma(g).
\]

Let $\theta$ and $\sigma$ act on $\mathfrak{g}$ by their differentials and denote
\[
\mathfrak{g}^\sigma \coloneqq \{a \in \mathfrak{g} \mid \sigma(a) = a\} = \{a \in \mathfrak{g} \mid \theta(a) = -a\}.
\]

Note that we have the standard two-sided $H \times H$ action on both $G$ and $\mathfrak{g}$, under the symmetrization map corresponding to adjoint $H$-action on both $G^\sigma$ and $\mathfrak{g^\sigma}$. Also, symmetrization map induced an injection
\[
G/H \xhookrightarrow{s} G^\sigma.
\]
\end{definition}

\begin{theorem}
For any connected symmetric pair $(G, H, \theta)$, we have $\mathcal{O}(G)^{H \times H} \subset \mathcal{O}(G)^\sigma$.
\end{theorem}

\begin{proof}
Consider the multiplication map $H \times G^\sigma$. This is a smooth map of relative dimension $0$ at the point $(1, 1)$, so \'etale at $(1, 1)$ and its image $HG^\sigma$ contains an open neighborhood of $1$ in $G$. Combining with the projection map
\[
H \times G^\sigma \rightarrow G \rightarrow G/H,
\]
we see that the image $HG^\sigma$ is also dense in $G/H$. Thus $HG^\sigma H$ is dense in $G$. Clearly we can see that
\[
\mathcal{O}(HG^\sigma H)^{H \times H} \subset \mathcal{O}(HG^\sigma H)^\sigma,
\]
and taking closure on both sides, we get
\[
\mathcal{O}(G)^{H \times H} \subset \mathcal{O}(G)^\sigma.
\]
\end{proof}

\begin{corollary}
For any connected symmetric pair $(G, H, \theta)$ and any closed $H \times H$ orbit $\Delta \subset G$, we have $\sigma(\Delta) = \Delta$.
\end{corollary}

\begin{proof}
Denote $K \coloneqq H \times H$. Consider the action of $2$-element group $(1, \tau)$ on $K$ by $\tau(h_1, h_2) \coloneqq \big(\theta(h_2), \theta(h_1)\big)$. This defines a semi-direct product $\tilde K \coloneqq (1, \tau) \ltimes K$. Extend the two-sided action of $K$ to $\tilde K$ by the anti-involution, i.e.
\[
(\tau, h_1, h_2)g \coloneqq h_1\sigma(g)h_2^{-1}.
\]

Now we look at the categorical quotients (see Definition \ref{cat_quo_def}), and by Theorem \ref{cat_quo_thm}
\[
G \sslash K = \mathrm{Spec}\mathcal{O}(G)^K, \quad G \sslash \tilde K = \mathrm{Spec}\mathcal{O}(G)^{\tilde K}.
\]

Since $\mathcal{O}(G)^{H \times H} \subset \mathcal{O}(G)^\sigma$ from the previous lemma, $G \sslash K = G \sslash \tilde K$. Let $\Delta$ be a closed $K$-orbit, and $\tilde\Delta \coloneqq \Delta \cup \sigma(\Delta)$. Let $a \coloneqq \pi_G(\tilde\Delta) \subset G \sslash \tilde K$. Since $\Delta$ and $\sigma(\Delta)$ are in the same $\tilde K$-orbit, $a$ is a single point. So $\pi_G^{-1}(a)$ contains a unique closed $G$-orbit by Theorem \ref{cat_quo_thm}. Therefore, $\Delta = \sigma(\Delta)$.
\end{proof}

\begin{corollary}\label{closed_orb}
Let $(G, H, \theta)$ be a connected symmetric pair. Let $g \in G(F)$ such that $HgH$ is a closed orbit in $G$. Suppose the Galois cohomology $H^1\big(F, (H \times H)_g\big)$ is trivial. Then $\sigma(g) \in H(F)gH(F)$.
\end{corollary}

\begin{proof}
For fixed $g \in G(F)$, define
\[
A \coloneqq (H \times H)_g, \quad Y \coloneqq \{(h_1, h_2) \in H \times H \mid h_1\sigma(g)h_2^{-1} = g \}.
\]
Then $Y$ is an $A$-torsor (an $A$-set isomorphic to $A$), where $A$ acts on $Y$ by standard left multiplication. Now it suffices to show that $Y(F) \neq \emptyset$. This is to say that $H^1(F, A) = 0$ guaranteed a rational point in any $A$-torsor. We can define a map using a fixed $(h_1, h_2) \in Y$:
\[
f: \mathrm{Gal}(\bar F / F) \rightarrow A, \quad \varphi \mapsto \big(\varphi(h_1)h_1^{-1}, h_2\varphi(h_2^{-1})\big).
\]

In fact, this is a cocycle in $C^1(F, A)$, which consists of crossed homomorphisms, and one can easily calculate that
\[
f(\varphi_1\varphi_2) = f(\varphi_1)\varphi_1f(\varphi_2).
\]

For Galois cohomology on low dimensions (see \cite{Serre} Chapter I, 2.3), we have $H^0(G, A) = A^G = A(F)$ and $H^1(F, A)$ is the set of equivalence classes of crossed homomorphisms, which corresponds bijectively to the isomorphism class of $A$-torsors. Since $ H^1(F, A)$ is trivial, $f$ has to be a cochian, i.e.
\[
\exists (k_1, k_2) \in A,  f(\varphi) = \big(\varphi(k_1)k_1^{-1}, k_2\varphi(k_2^{-1})\big), \forall \varphi \in \mathrm{Gal}(\bar F / F).
\]

So $(k_1^{-1}h_1, k_2^{-1}h_2) \in Y(F)$, as desired.
\end{proof}

\subsection{Closed Orbits and Semi-simplicity}

In this section, we want to prove that $HgH$ is a closed orbit iff $x = s(g) \in G^\sigma$ is semi-simple as an element of $G$.

\begin{lemma}
Let $(G, H, \theta)$ be a symmetric pair. There exists a $G$-invariant $\theta$-invariant non-degenerate symmetric bilinear form $B$ on $\mathfrak{g}$. In particular, $\mathfrak{g} = \mathfrak{h} \oplus \mathfrak{g}^\sigma$ is an orthogonal direct sum with respect to $B$.
\end{lemma}

\begin{proof}
Suppose $\mathfrak{g}$ is semi-simple. Let $B$ be the Killing form on $\mathfrak{g}$. Let $\alpha, \gamma \in \mathfrak{h}$ and $\beta \in \mathfrak{g}^\sigma$. Since
\begin{align*}
\sigma\big(\mathrm{ad}(\alpha)\mathrm{ad}(\beta)\gamma\big) &= \sigma(\alpha\beta\gamma - \alpha\gamma\beta - \beta\gamma\alpha + \gamma\beta\alpha) \\
&= \gamma\beta\alpha - \beta\gamma\alpha - \alpha\gamma\beta + \alpha\beta\gamma \\
&= \mathrm{ad}(\alpha)\mathrm{ad}(\beta)\gamma,
\end{align*}
so $\mathrm{ad}(\alpha)\mathrm{ad}(\beta)\mathfrak{h} \subset \mathfrak{g}^\sigma$. Similarly, $\mathrm{ad}(\alpha)\mathrm{ad}(\beta)\mathfrak{g}^\sigma \in \mathfrak{h}$. Thus $\mathrm{Tr}\big(\mathrm{ad}(\alpha)\mathrm{ad}(\beta)\big) = 0$, i.e. $\mathfrak{h}$ is orthogonal to $\mathfrak{g}^\sigma$ with respect to $B$.

If $\mathfrak{g}$ is not semi-simple, let $\mathfrak{g} = \mathfrak{g}' \oplus \mathfrak{z}$ such that $\mathfrak{g}'$ is semi-simple and $\mathfrak{z}$ is the center. The decomposition is invariant under the action of any elements in $\mathrm{Aut}(\mathfrak{g})$, therefore $\theta$-invariant. Then it follows from the semi-simple case by taking the Killing form on each semi-simple component.
\end{proof}

\begin{definition}
A set of three linearly independent elements $(h, e, f)$ is called an $\mathfrak{sl}_2$-triple if the bracket relations are satisfied:
\[
[h, e] = 2e, \quad [h, f] = -2f, \quad [e, f] = h.
\]

An $\mathfrak{sl}_2$-triple is called \textbf{normal} if $e, f \in \mathfrak{g}^\sigma$ and $h \in \mathfrak{h}$.
\end{definition}

\begin{lemma}\label{sl_lemma}
Let $(G, H, \theta)$ be a symmetric pair. Let $x \in \mathfrak{g}^\sigma$ be a nilpotent element. Then

\begin{enumerate}
\item $0 \in \overline{\mathrm{Ad}(H)x}$.
\item There is a correspondence between the set of all $H$-orbits in $\mathfrak{n} - \{0\}$ and the set of all $H$-conjugacy classes of normal $\mathfrak{sl}_2$-triple.
\end{enumerate}
\end{lemma}

\begin{proof}
\begin{enumerate}
\item By Jacobson-Morozov Theorem (\cite{Jacobson} Chapter III, Theorem 10 and 17), any nilpotent element $e = x \in \mathfrak{g}$ can be included into an $\mathfrak{sl}_2$-triple $(h, e, f)$ in the following way:

    Let $h' \coloneqq \frac{h + \theta(h)}{2}$, then
    \[
    [h', e] = \frac{1}{2}[h, e] + \frac{1}{2}[\theta(h), e] = e + \frac{1}{2}\theta\big([h, \theta(e)]\big) = e + \frac{1}{2}\theta\big([h, -e]\big) = e - \theta(e) = 2e.
    \]

    So by Morozov Lemma (\cite{Jacobson} Chapter III, Lemma 7), $e$ and $h'$ can be completed to an $\mathfrak{sl}_2$-triple $(h', e, f')$. Let $f'' \coloneqq \frac{f' - \theta(f')}{2}$. Then we get an $\mathfrak{sl}_2$-triple $(h', e, f'')$ with $e$ nilpotent, $h' \in \mathfrak{g}^\theta = \mathfrak{h}$, and $f'' \in \mathfrak{g}^\sigma$.

    We would want to use the exponential map $\mathrm{exp}: \mathfrak{n} \rightarrow \mathcal{U}$ (which is $\sigma$-equivariant and intertwines the adjoint action with conjugation action) of this $\mathfrak{sl}_2$-triple to achieve a homomorphism $SL_2 \rightarrow G$ and
    \begin{align}\label{limit_eq}
    \lim_{t \rightarrow 0} \mathrm{Ad}(\begin{bmatrix}t & 0 \\ 0 & t^{-1}\end{bmatrix})x = \lim_{t \rightarrow 0}\begin{bmatrix}0 & 2t \\ 0 & 0\end{bmatrix} = 0.
    \end{align}

    However, we are working over finite fields where the exponential map cannot be defined, but we can still use the idea of an exponential map. For any characteristics $p$ larger than the order of $e$ and $f$, we denote
    \[
    \exp(te) \coloneqq 1 + te + \frac{1}{2!}t^2e^2 + \cdots, \quad \exp(tf) \coloneqq 1 + tf + \frac{1}{2!}t^2f^2 + \cdots,
    \]

    which is well-defined since $e$ and $f$ are nilpotent and the characteristic is large enough. Here we are thinking of an embedding $\mathbf{G} \hookrightarrow \mathbf{Mat}$ into matrix scheme. Now consider the element,
    \[
    \alpha_t \coloneqq \exp(te)\exp(-t^{-1}f)\exp(te), \quad \theta(\alpha_t) = \alpha_{-t}.
    \]
    The Construction of such an element comes from the following calculation in $SL_2$:
    \[
    \begin{bmatrix}1 & t \\ 0 & 1\end{bmatrix}\begin{bmatrix}1 & 0 \\ -t^{-1} & 1\end{bmatrix}\begin{bmatrix}1 & t \\ 0 & 1\end{bmatrix} = \begin{bmatrix}0 & t \\ -t^{-1} & 0\end{bmatrix}, \quad \begin{bmatrix}0 & t \\ -t^{-1} & 0\end{bmatrix}\begin{bmatrix}0 & 1 \\ -1 & 0\end{bmatrix} = \begin{bmatrix}t & 0 \\ 0 & t^{-1}\end{bmatrix}.
    \]

    We claim the following:
    \begin{enumerate}
    \item $\exp(te), \exp(tf) \in G$.
    \item $\alpha_t\alpha_1 \in H$, i.e.
    \begin{equation}\label{eq_1}
    \theta(\alpha_t\alpha_1) = \alpha_{-t}\alpha_{-1} = \alpha_t\alpha_1 \Longleftrightarrow \alpha_t^2 = \alpha_{-1}^2.
    \end{equation}
    \end{enumerate}

    We have seen that if there exists an exponential map both claims are true. $\mathbf{G}$ is defined by a finite collection of polynomials, so $\exp(te) \in G$ is determined by whether they satisfy these polynomials. The relation in \eqref{eq_1} is a polynomial as well.

    \begin{lemma}
    Let $V \subset \mathbb{A}^N$ an affine scheme over $Z$, $f \in \mathbb{Q}[x_1, x_2, \cdots, x_N]$ a regular function. If $f(V(\mathbb{C})) = 0$, then there is $p_0$ such that $\forall p > p_0$ we have $f(V(\mathbb{\bar F}_p)) = 0$.
    \end{lemma}

    This lemma is clear by taking $p_0$ large enough that all the coefficients of $f$ make sense over $\mathbb{\bar F}_p$. Therefore, for large enough $p_0$, our claims hold.

    Then $0 \in \overline{\mathrm{Ad}(H)x}$ because of equation \eqref{limit_eq}.

\item Assume that $(h, e, f)$ and $(h_1, e, f_1)$ are two normal $\mathfrak{sl}_2$-triples. Let $y = h_1 - h$, so $[e, y] = 0$ and $y \in \mathfrak{h}$. Since $[h, e] = 2e$, for any $x \in \mathfrak{h}_e$, we have $[x, h] \in \mathfrak{h}_e$ since
\[
\big[[x, h], e\big] = -\big[[h, e], x\big] - \big[[e, x], h\big] = -[2e, x] = 0.
\]

So $\mathfrak{h}_e$ is stable under $\mathrm{ad}(h)$. From representation theory of $\mathfrak{sl}_2$, we see that $h$ is semi-simple, and $\mathrm{ad}(h)$ is diagonalizable on $\mathfrak{h}_e$. In fact, its eigenvalues are non-negative integers (highest weights). Let $\mathfrak{h}_e^+ \subset \mathfrak{h}_e$ be the subspace spanned by all eigenvectors with strictly positive integers eigenvalues. Then
\[
\mathfrak{h}_e^+ = \mathfrak{h}_e \cap [e, \mathfrak{g}],
\]

and $y = [e, f_1 - f] \in \mathfrak{h}_e^+$. Let $U$ be the subgroup of $G$ corresponding to $\mathrm{ad}\mathfrak{h}_e^+$, then $U$ is a unipotent subgroup. Now $Uh \subset h + \mathfrak{h}_e^+$ is closed, and $Uh$ is open in $h + \mathfrak{h}_e^+$. Therefore $Uh = h + \mathfrak{h}_e^+ \subset Hh$, i.e. there is $a \in H$ such that $a(h, e, f) = (h_1, e, af)$. But an $\mathfrak{sl}_2$-triple is determined by the first two elements (see Corollary 3.5 on p.984 of \cite{Kostant2} by replacing $h$ with $x$), so $af = f_1$. This gives a well-defined map $He \mapsto H(h, e, f)$ we desired. In fact, this is a one-to-one correspondence.
\end{enumerate}
\end{proof}

\begin{lemma}\label{conjugate_lemma}
Let $(x, e, f)$ and $(x_1, e_1, f_1)$ be normal $\mathfrak{sl}_2$-triples. Then they are $H$-conjugate iff $x$ and $x_1$ are $H$-conjugate.
\end{lemma}

\begin{proof}
We may assume that $x = x_1$. Let $\mathfrak{g}_x$, $\mathfrak{h}_x$, $\mathfrak{g}^\sigma_x$ be centralizer of $x$ in the corresponding algebras. Let $\mathfrak{g}_1$, $\mathfrak{h}_1$, $\mathfrak{g}^\sigma_1$ be respectively the spaces of all $y$ in $\mathfrak{g}$, $\mathfrak{h}$, $\mathfrak{g}^\sigma$ such that $[x, y] = 2y$. Clearly,
\[
\mathfrak{g}_x = \mathfrak{h}_x + \mathfrak{g}^\sigma_x, \quad \mathfrak{g}_1 = \mathfrak{h}_1 + \mathfrak{g}^\sigma_1.
\]
Note that $[\mathfrak{g}_x, \mathfrak{g}_1] \subset \mathfrak{g}_1$, so $[\mathfrak{h}_x, \mathfrak{g}^\sigma_1] \subset \mathfrak{g}^\sigma_1$, and $\mathfrak{g}_1^\sigma$ is a $H_x$-module. Define
\[
V \coloneqq \{y \in \mathfrak{g}_1^\sigma \mid [\mathfrak{h}_x, y] = \mathfrak{g}_1^\sigma\},
\]
which is a Zariski open set. By representation theory of $\mathfrak{sl}_2$, any weight vector of weight $2$ for $x$ is the image of a zero weight vector under the action $\mathrm{ad}x$. So
\[
[e, \mathfrak{g}_x] = \mathfrak{g}_1 \Longrightarrow [e, \mathfrak{h}_x] = \mathfrak{g}_1^\sigma, \quad [e, \mathfrak{g}_x^\sigma] = \mathfrak{h}_1.
\]
Thus, $e \in V$, and similarly $e_1 \in V$. For any $y \in V$, because of the definition of $V$, tangent plane of $H_xy$ is the same of tangent plane of $V$ at $y$. So $H_xy \in V$ is an open subset, and $V$ is connected since it is open. Thus $V$ consists of a single $H_x$-orbit, i.e. $e_1 = ae$ for some $a \in H_x$. Then $a(x, e, f) = (x, e_1, af)$, and $af = f_1$ because this triple is determined by the first two elements.
\end{proof}

\begin{proposition}\label{closesd_to_ss}
Let $(G, H, \theta)$ be a symmetric pair. Let $g \in G$ such that $HgH$ is a closed $H \times H$-orbit. Let $x = s(g)$. Then $x$ is semi-simple as an element of $G$, and $\mathrm{Ad}(H)x$ is a closed $H$-orbit.
\end{proposition}

\begin{proof}
Clearly the image of $HgH$ in $G/H$ is closed. Since the symmetrization map, viewed as map from $G/H$ to $G$, is a closed embedding, it follows that the $H$-orbit of $x$ is closed. Let $x = x_sx_u$ be the Jordan decomposition of $x$. The uniqueness of Jordan decomposition implies that both $x_s, x_u \in G^\sigma$. We claim that $x_s \in \overline{\mathrm{Ad}(H)x}$ for any $x$, and since $\mathrm{Ad}(H)x$ is closed, $x = \mathrm{Ad}(h)x_s$ is semi-simple. Now we prove our claim in the following steps:

\begin{itemize}
\item If $x_s = 1$, it follows from the proof of Lemma \ref{sl_lemma}.

\item If $x_s \in Z(G)$, then this follows from the first case since conjugation acts trivially on $Z(G)$.
\item For general $x$, $x \in G_{x_s}$ and $G_{x_s}$ is $\theta$-invariant. Now this follows from the previous case since $x \in Z(G_{x_s})$.
\end{itemize}
\end{proof}

\begin{theorem}\label{ss_to_closed}
Let $x \in G^\sigma$ be semi-simple as an element of $G$. Then $\mathrm{Ad}(H)x$ is a closed orbit.
\end{theorem}

\begin{proof}
To understand $H$-orbit of $x$, we consider $G$-orbit of $x$ by conjugation action, which is closed since $x$ is semi-simple. We claim that $\mathrm{Ad}(H)x = \mathrm{Ad}(G)x \cap G^\sigma$, which is closed. We consider the tangent spaces of these subsets at the point $x$:

\begin{align*}
T_xG^\sigma &= \{\alpha \in \mathfrak{g} \mid \mathrm{Ad}(x)\sigma(\alpha) = \alpha\}, \\
T_x\mathrm{Ad}(G)x &= \{\mathrm{Ad}(x)(\alpha) - \alpha: \alpha \in \mathfrak{g}\} \\
T_x\mathrm{Ad}(H)x &= \{\mathrm{Ad}(x)(\alpha) - \alpha: \alpha \in \mathfrak{h}\}
\end{align*}

Then we have
\[
T_xG^\sigma \cap T_x\mathrm{Ad}(G)x = \{\mathrm{Ad}(x)(\alpha) - \alpha: \alpha \in \mathfrak{g}, \mathrm{Ad}(x)\big(\alpha + \sigma(\alpha)\big) = \alpha + \sigma(\alpha)\}
\]

Since $\mathfrak{g} = \mathfrak{h} \oplus \mathfrak{g}^\sigma$, we can write $\alpha = \alpha_\mathfrak{h} + \alpha_\sigma$ according to this decomposition. Then $\alpha + \sigma(\alpha) = 2\alpha_\sigma$, which means that $\mathrm{Ad}(x)(\alpha_\sigma) - \alpha_\sigma = 0$. Thus,
\[
T_xG^\sigma \cap T_x\mathrm{Ad}(G)x = \{\mathrm{Ad}(x)(\alpha_\mathfrak{h}) - \alpha_\mathfrak{h}: \alpha \in \mathfrak{g}, \mathrm{Ad}(x)\alpha_\sigma = \alpha_\sigma\} = T_x\mathrm{Ad}(H)x.
\]

This shows that $\mathrm{Ad}(H)x \subset \mathrm{Ad}(G)x \cap G^\sigma$ as an open subset. Pick $y \in \overline{\mathrm{Ad}(H)x} - \mathrm{Ad}(H)x \subset \mathrm{Ad}(G)x \cap G^\sigma - \mathrm{Ad}(H)x$. So $y$ is a semi-simple element of $G$ as well, and $\mathrm{Ad}(H)y \subset \mathrm{Ad}(G)x \cap G^\sigma$ as an open subset. But then $\mathrm{Ad}(H)x \cap \mathrm{Ad}(H)y \neq \emptyset$, a contradiction.
\end{proof}

\subsection{Categorical Quotient}
\begin{definition}\label{cat_quo_def}[See \cite{Drezet} Definition 2.11]
Let an algebraic group $G$ act on an algebraic variety $X$. A pair consisting of an algebraic variety $Y$ and a $G$-invariant morphism $\pi: X \rightarrow Y$ is called a \textbf{categorical quotient} of $X$ by the action of $G$ if for any pair $(Y', \pi')$ there exists a unique morphism $\varphi: Y \rightarrow Y'$ such that $\pi' = \varphi \circ \pi$. Clearly, if such a quotient exists, it is unique up to a canonical isomorphism. We denote this quotient by $(G \sslash X, \pi_X)$.
\end{definition}

\begin{theorem}\label{cat_quo_thm}
Let an algebraic group $G$ act on an affine variety $X$. Then:
\begin{enumerate}
    \item The categorical quotient $X \sslash G$ exists. In fact, $X \sslash G = \mathrm{Spec}\mathcal{O}(X)^G$.
    \item Every fibre of the quotient map $\pi_X$ contains a unique closed orbit.
\end{enumerate}
\end{theorem}

\begin{proof}
\begin{enumerate}
\item This is proved in \cite{Drezet} Theorem 2.16, and in \cite{Newstead} Theorem 3.5. Essentially we just need to check $\mathrm{Spec}\mathcal{O}(X)^G$ satisfies the universal property of categorical quotient.

\item This is proved in \cite{Drezet} Lemma 2.13. The idea is that the orbit of minimal dimension is a closed orbit, and it is unique because distinct closed orbits are disjoint.
\end{enumerate}
\end{proof}

\begin{theorem}\label{finite_to_one}
Let $(G, H, \theta)$ be a symmetric pair. There is the standard two-sided action of $H \times H$ on $G$. Then Every fibre of the quotient map $\pi_G: G \rightarrow G \sslash H \times H$ contains finitely many orbits.
\end{theorem}

\begin{proof}
Recall that the symmetrization map induced an injection $G/H \xhookrightarrow{s} G^\sigma$. So it suffices to prove that every fibre of the quotient map $\pi_X$ contains finitely many orbits for $X = G^\sigma$ with $H$ acting by conjugation.

Suppose $\mathrm{Ad}(H)x$ is the unique closed orbit in a fibre of $\pi_X$ with $x$ semi-simple, and $\mathrm{Ad}(H)y$ is another orbit with $\pi_X(x) = \pi_X(y)$. Let $y = y_sy_n$ be its Jordan decomposition. We have seen in the proof of \ref{closesd_to_ss} that $y_s \in \overline{\mathrm{Ad}(H)y}$, so by the definition of categorical quotient,
\[
\pi_X(x) = \pi_X(y) = \pi_X(y_x) \Rightarrow \mathrm{Ad}(H) x = \mathrm{Ad}(H) y_s.
\]

This shows that two elements have the same image in the categorical quotient iff their semi-simple part are $H$-conjugate. We claim that two elements with the same semi-simple part are $H$-conjugate iff their nilpotent parts are $H_{x_s}$ conjugate. This follows from the uniqueness of Jordan decomposition:
\[
x = hyh^{-1} \Leftrightarrow x_sx_n = (hy_sh^{-1})(hy_nh^{-1}) \Leftrightarrow h \in H_{x_s}, x_n = hy_nh^{-1}.
\]

Denote $\mathcal{N}_{x_s}$ to be all the nilpotent elements that commute with $x_s$. Then the intersection $\mathrm{Ad}(H)y \cap x_s\mathcal{N}_{x_s}$ is a single $H_{x_s}$-orbit in $\mathcal{N}_{x_s}$. Therefore, the correspondence
\[
\mathrm{Ad}(H)y \rightarrow \mathrm{Ad}(H)y \cap x_s\mathcal{N}_{x_s}
\]

sets up a bijection between the set of all $H$-orbits in this fibre of $\pi_X$ and all $H_{x_s}$-orbit in $\mathcal{N}_{x_s}$. If we look at a new group $G_{x_s}^\sigma$, with its corresponding subsets, $H_{x_s} = G_{x_s}^\theta$, then it corresponds to the following lemma for any symmetric pairs combined with the logarithm map as inverse to the exponential map we constructed in the proof of Lemma \ref{sl_lemma}.
\end{proof}

\begin{lemma}
Let $(G, H, \theta)$ be a symmetric pair. Then there are finitely many $H$-orbit in $\mathfrak{n} - 0$.
\end{lemma}

\begin{proof}
Let $X$ be the set of $h \in \mathfrak{h}$ that can be embedded in a normal $\mathfrak{sl}_2$-triples. By the proof of Lemma \ref{sl_lemma}, $H$ orbits in $\mathfrak{n} - 0$ corresponds to $H$-conjugacy classes of normal $\mathfrak{sl}_2$-triples, and $X$ is stable under $H$-conjugation. Therefore, by Lemma \ref{conjugate_lemma} it suffices to show that there is a finite number of $H$-orbits in $X$.

Let $\mathfrak{h}_k$ be a Cartan subalgebra of $\mathfrak{h}$, and $\mathfrak{k}$ be a Cartan subalgebra of $\mathfrak{g}$ containing $\mathfrak{h}_k$. Let $\Delta$ be the set of roots of $\mathfrak{k}$ acting on $\mathfrak{g}$. Since elements $x \in X$ can be embedded in a normal $\mathfrak{sl}_2$-triples, they are semi-simple. So $x$ is $H$-conjugate to some $y \in \mathfrak{h}_k$. The eigenvalues of $\mathrm{ad}x$ are integers with norm $\leq \dim G$, then it follows that $|(y, \phi)|$ is an integer between zero and $\dim G$ for all $\phi \in \Delta$.

Now $y \in [\mathfrak{g}, \mathfrak{g}]$ since $x \in [\mathfrak{g}, \mathfrak{g}]$. But any element in $[\mathfrak{g}, \mathfrak{g}] \cap \mathfrak{k}$ is determined by the values $(y, \phi)$ for all $\phi \in \Delta$, where $|(y, \phi)|$ has finitely many options. So there are only finitely many elements in $[\mathfrak{g}, \mathfrak{g}] \cap \mathfrak{k}$ which is $H$-conjugate to element in $X$. Hence, there are finitely many $H$-orbits in $X$.
\end{proof}

\subsection{Dense Subset of Semi-simple Elements}
We use one of the main results in \cite{Aizenbud2} to find a subset of semi-simple elements in $G^\sigma$.

\begin{theorem}
Let $(G, H, \theta)$ be a symmetric pair. We identify $T^*G$ with $G \times \mathfrak{g}^*$ and let
\[
S \coloneqq \{(g, \alpha) \in G \times \mathfrak{g}^* \mid \alpha \text{ nilpotent }, \langle\alpha, \mathfrak{h}\rangle = 0, \langle\alpha, \mathrm{Ad}(g)\mathfrak{h}\rangle = 0\}.
\]

We have $\dim S = \dim G$.
\end{theorem}

\begin{proof}
See \cite{Aizenbud2} Theorem B with the same notation. Their result is for groups over real numbers, but the same proof works for algebraic closure of finite fields.
\end{proof}

\begin{corollary}\label{Zariski_open_subset}
Let $U \coloneqq \big\{g \in G \mid S \cap T_g^*G = \{(g, 0)\}\big\}$. Note that $Z$ is Zariski open since $S$ is conic and closed. In fact, $U$ is a Zariski open dense subset of $G$.
\end{corollary}

\begin{proof}
See \cite{Aizenbud2} Corollary C with the same notation.
\end{proof}

We need to find out which set $S' \subset T^*G^\sigma$ corresponds to $S$ under the co-differential of the symmetrization map $d^*s: T^*G^\sigma \rightarrow T^*G$.

\begin{lemma}\label{open_subset_ss}
Define $\sigma_x \coloneqq \mathrm{Ad} \circ \sigma$, which is a new anti-involution, and
\begin{align*}
S' \coloneqq& \{(x, \beta) \in G^\sigma \times (\mathfrak{g}^*)^{\sigma_x^*} \mid \beta \text{ nilpotent }, \mathrm{Ad}^*(x)\beta = \beta\}, \\
V \coloneqq& \big\{g \in G \mid S' \cap T_x^*G^\sigma = \{(x, 0)\}\big\}.
\end{align*}
Then $s(U) \subset V$. It is not hard to see that $V$ is consisting of semi-simple elements of $G$, since the nilpotent part in the Jordan decomposition commutes with $x$ and it is fixed by the anti-involution as shown in Proposition \ref{closesd_to_ss}.
\end{lemma}

\begin{proof}
In the proof of \ref{closesd_to_ss}, we have seen that
\[
T_xG^\sigma = \{\alpha \in \mathfrak{g} \mid \mathrm{Ad}(x)\sigma(\alpha) = \alpha\} = \mathfrak{g}^{\sigma_x} \Rightarrow T_x^*G^\sigma = (\mathfrak{g}^*)^{\sigma_x^*}.
\]

By easy calculation, we can see that
\begin{align*}
d_gs(\alpha) &= \alpha + \mathrm{Ad}(x)\sigma(\alpha) = \alpha + \sigma_x(\alpha), \\
d_g^*s(\beta) &= \beta + \sigma_x^*(\beta),
\end{align*}

We claim that if $(x, \beta) \in S'$, then $\big(g, d_g^*(\beta)\big) \in S$. From the formula above, we see that on $(\mathfrak{g}^*)^{\sigma_x^*}$, co-differential $d_g^*s(\beta) = 2\beta$, which is again nilpotent. For any $\alpha \in \mathfrak{h}$, we have
\begin{align*}
\langle d_g^*(\beta), \mathrm{Ad}(g)\alpha\rangle &= \langle\beta, d_g(g\alpha g^{-1})\rangle \\
&= \langle\beta, g\alpha g^{-1} + \sigma_x(g\alpha g^{-1})\rangle \\
&= \langle\beta, g\alpha g^{-1} + x\sigma(g^{-1})\sigma(\alpha)\sigma(g)x^{-1})\rangle \\
&= \langle\beta, g\alpha g^{-1} + g\sigma(\alpha)g^{-1})\rangle \\
&= \langle\beta, 0\rangle = 0,
\end{align*}
and since $\mathrm{Ad}^*(x)\beta = \beta$,
\begin{align*}
\langle d_g^*(\beta), \alpha\rangle &= \langle\beta, d_g(\alpha)\rangle \\
&= \langle\beta, \alpha + \sigma_x(\alpha)\rangle \\
&= \langle\beta, \alpha + x\sigma(\alpha)x^{-1}\rangle \\
&= \langle\beta, \alpha - x\alpha x^{-1}\rangle \\
&= -\langle\beta, \mathrm{Ad}(x)(\alpha - x\alpha x^{-1})\rangle \\
&= -\langle\mathrm{Ad}^*(x)(\beta),\alpha - x\alpha x^{-1}\rangle \\
&= -\langle\beta, \alpha - x\alpha x^{-1}\rangle = 0.
\end{align*}

If there is a point $g \in U$ and $x = s(g) \notin V$, then there is $\beta \neq 0$ such that $(x, \beta) \in S' \cap T_x^*G^\sigma$. Now $\big(g, d_g^*(\beta)\big) = (g, 2\beta)\in S$, a contradiction.
\end{proof}

\section{An Extension of Gelfand's Trick}\label{ch_Gel}
To make use of fixed points of anti-involutions, we start by looking a simple case of anti-involution over matrix algebras. Then we extend the same idea to semi-simple algebras.

\subsection{Fixed Points of Anti-automorphism on Matrix Algebras}
\begin{lemma}[Upper Bound of Fixed Points of Anti-automorphism]
Let $M = M_n(\mathbb{C})$ be the matrix algebra of rank $n \geq 2$, and $\sigma: M \rightarrow M$ be an anti-automorphism, then $\dim M^\sigma \leq \frac{n(n + 1)}{2}$. Moreover, equal sign can be achieved when $\sigma$ is taking transpose.
\end{lemma}

\begin{proof}
Claim: $\sigma(A) = gA^Tg^{-1}$ for some $g \in \mathrm{GL}_n(\mathbb{C})$.

Since $\sigma$ is an anti-automorphism, $A \mapsto \sigma(A^T)$ is an automorphism of the matrix algebra $M = M_n(\mathbb{C})$:
\[
\sigma(A^T)\sigma(B^T) = \sigma(B^TA^T) = \sigma\big((AB)^T\big).
\]

By Skolem-Noether theorem, such automorphism is an inner automorphism, i.e. conjugation by an invertible matrix. Therefore,

\[
\sigma(A^T) = gAg^{-1} \Rightarrow \sigma(A) = gA^Tg^{-1}.
\]
Now we see that
\[
M^\sigma = \{A \in M \mid A = gA^Tg^{-1}\} = \{A \in M \mid Ag = gA^T\},
\]
for some $g \in \mathrm{GL}_n(\mathbb{C})$. In addition,
\[
Ag = gA^T \Leftrightarrow g^TA^T = Ag^T \Leftrightarrow A^T = (g^T)^{-1}Ag^T.
\]
If we denote $X = A g^T$, then $X^T = gA^T$, and
\[
(g^T)^{-1}X = A^T = g^{-1}X^T.
\]
Therefore, if we denote $h = g^{-1}$, then we have
\[
M^\sigma = \{A \in M \mid h^TX = hX^T, X = Ag^T\} = \{X \in M \mid h^TX = hX^T\}(g^T)^{-1}.
\]
Now we claim that for any $h \in \mathrm{GL}_n(\mathbb{C})$,
\[
\dim \{X \in M \mid h^TX = hX^T\} \leq \frac{n(n + 1)}{2}.
\]
Let us denote $h = (h_{ij})$ and $X = (x_{ij})$ as matrix entries. Then the equation $h^TX = hX^T$ is in fact a linear system of $n^2$ equations in terms of $n^2$ variables $x_{ij}$ with coefficients coming from $h_{ij}$. We label these $n^2$ equations by there corresponding position in the resulting matrix, i.e. $h^TX = hX^T \rightarrow (Eq_{ij})$ where $Eq_{ij}$ represents the $ij$-th equation.

We first look at the first column $Eq_{i1}$: on the LHS is the $i$-th column of $h$ multiplying the first column of $X$, while on the RHS is the $i$-th row of $h$ multiplying the first row of $X$,

\[
Eq_{i1}: \sum_{j = 1}^nh_{ji}x_{j1} = \sum_{j = 1}^nh_{ij}x_{1j}. \quad (i = 1, 2, \cdots, n)
\]

These $n$ equations only involve $(2n - 1)$variables $x_{1j}$ and $x_{j1}$, and only $x_{11}$ appears on both sides of the equations. Rearranging the equations gives us:

\[
Eq_{i1}: (h_{1i} - h_{i1})x_{11} + \sum_{j = 2}^nh_{ji}x_{j1} - \sum_{j = 2}^nh_{ij}x_{1j} = 0. \quad (i = 1, 2, \cdots, n)
\]
If we rewrite it in matrix form, it will look like
\[
\begin{bmatrix}
(h_{1i} - h_{i1}) & h_{ji} & -h_{ij} \end{bmatrix}\begin{bmatrix}
x_{11} \\ x_{j1} \\ x_{1j}
\end{bmatrix} = 0, \quad i \geq 1, j \geq 2.
\]

Observe that the coefficients of $x_{12}, \cdots, x_{1n}$ comes from columns of $h$ except the first column, i.e. $(-h_{ij}), j \neq 1$, meaning that these $(n - 1)$ columns are linearly independent, and so the matrix corresponding to $Eq_{i1}$'s is of rank at least $(n - 1)$, i.e. at least $(n - 1)$ linearly independent equations in these $n$ equations.

Next, we look at $Eq_{i2}$, $i = 2, \cdots, n$. We introduce $(n - 1)$ new equations and $2n - 3$ variables $x_{2j}$ and $x_{j2}$ with $j \geq 2$.

\[
Eq_{i2}: \sum_{j = 1}^nh_{ji}x_{j2} = \sum_{j = 1}^nh_{ij}x_{2j}. \quad (i = 2, 3, \cdots, n)
\]
Rearranging the equations gives us:

\[
Eq_{i2}: (h_{2i} - h_{i2})x_{22} + \sum_{k \neq 2}^nh_{ki}x_{k2} - \sum_{k \neq 2}^nh_{ik}x_{2k} = 0. \quad (i = 2, 3, \cdots, n)
\]
If we rewrite it in matrix form, it will look like

\[
\begin{bmatrix}
(h_{1i} - h_{i1}) & h_{ji} & -h_{ij} & 0 & 0 & 0 \\
0 & h_{1i} & -h_{i1} & h_{2i} - h_{i_2} & h_{ki} & -h_{ik}\end{bmatrix}\begin{bmatrix}
x_{11} \\ x_{j1} \\ x_{1j} \\ x_{22} \\ x_{k2} \\ x_{2k}
\end{bmatrix} = 0, \quad i \geq 1, j \geq 2, k \geq 3.
\]

Now we focus on the coefficients of $x_{1j}$ and $x_{2k}$, we have a matrix $\begin{bmatrix}
-h_{ij} & 0 \\ -h_{i1} & -h_{ik}
\end{bmatrix}$ of rank at least $2n - 3$, since

\[
Rank\begin{bmatrix}
A & 0 \\ B & C
\end{bmatrix} \geq Rank(A) + Rank(C).
\]

From the above two steps and by induction, one can easily see that by looking at the equations $Eq_{ij}, i \geq j$, i.e. the lower triangle of the resulting matrix, we can find at least $\frac{n(n - 1)}{2}$ linearly independent equations among them. Hence,
\[
\dim M^\sigma = \dim \{X \in M \mid h^TX = hX^T\} \leq n^2 - \frac{n(n - 1)}{2} = \frac{n(n + 1)}{2}.
\]
\end{proof}

In the case of anti-involution, the situation is much simpler since where are not many possible anti-involutions on matrix algebras.

\begin{lemma}\label{matrix_anti_inv_lemma}
Let $M = M_n(\mathbb{C})$ be the matrix algebra of rank $n \geq 2$, and $\sigma: M \rightarrow M$ be an anti-involution, then $\dim M^\sigma = \frac{n(n + 1)}{2}$ or $\frac{n(n - 1)}{2}$ which corresponds to symmetric matrices or skew-symmetric matrices.
\end{lemma}

\begin{proof}
We have seen that $\sigma(A) = gA^Tg^{-1}$ for some $g \in \mathrm{GL}_n(\mathbb{C})$, then

\[
A = \sigma^2(A) = g(gA^Tg^{-1})^Tg^{-1} = g(g^T)^{-1}A^Tg^Tg^{-1} = g(g^T)^{-1}A^T\big(g(g^T)^{-1}\big)^{-1}, \forall A.
\]

This means that $g(g^T)^{-1}$ commutes with all $A$, so it is a scaler matrix, i.e.

\[
g = cg^T = c^2g \Longrightarrow c = \pm 1.
\]

If $c = 1$, then $M^\sigma = \{A \in \mathrm{GL}_n(\mathbb{C}) \mid gA^T = Ag^T\}$, same size as all symmetric matrices with dimension $\frac{n(n + 1)}{2}$; if $c = -1$, then $M^\sigma$ is same size as all skew-symmetric matrices with dimension $\frac{n(n - 1)}{2}$.
\end{proof}

\subsection{Fixed Points of Anti-involution on Semi-simple Algebras}

Since semi-simple algebras are direct sum of matrix algebras, points not fixed by an anti-involution only comes from matrix algebra of rank $\geq 2$. Therefore, we can have an estimate of the number of rank-one matrix algebras using the dimension of fixed points of an anti-involution.

\begin{lemma}[Lower Bound of Rank-one Matrix Algebras]\label{multi_one_lemma}
Suppose $A$ is a semi-simple algebra over $\mathbb{C}$, and $\sigma: A \rightarrow A$ is an anti-automorphism with fixed points $A^\sigma$ of co-dimension $L = \epsilon\dim A$, then $A \cong \prod M_{n_i}(\mathbb{C})$, and with the number of rank $1$ matrix algebras $\big|\{n_i: n_i  = 1\}\big| \geq (1 - 4\epsilon)\dim A$.
\end{lemma}

\begin{proof}
First by Artin-Wedderburn theorem, $A$ decomposes into a product of simple algebras, and since the only division algebra over $\mathbb{C}$ is $\mathbb{C}$ itself,

\[
A \cong \prod A_i, \quad A_i \cong M_{n_i}(D_i) = M_{n_i}(\mathbb{C}).
\]

Define $\varphi: A \rightarrow A$ by $\sigma$ following by taking transpose in each matrix algebra. Then $\varphi$ is now an automorphism. Since $A_i$ is an ideal of $A$, $\varphi(A_i)$ is an ideal of $A$ as well, then $\varphi(A_i) \cap A_j$ will be an ideal of $A_j$ which is a simple algebra. So,
\[
\varphi(A_i) \cap A_j = 0 \text{ or } A_j.
\]
So the image $\varphi(A_i)$ must be isomorphic to a matrix algebra of the same rank. Therefore, $\varphi$ factors to each component of matrix algebras with possibly a rearrange of indices. Similarly, since $\varphi$ comes from $\sigma$ and taking transpose, $\sigma$ factors to each component of matrix algebras with possibly a rearrange of indices, i.e. we have two possible situations:
\begin{itemize}
\item either $\sigma(A_i) = A_i$, and the restriction $\sigma_i: A_i \rightarrow A_i$ is again an anti-automorphism. By the previous lemma, we see that if $n_i \geq 2$, then
    \[
    \dim A^{\sigma_i} \leq \frac{n_i(n_i + 1)}{2}.
    \]

\item or with a possible rearrange of indices
    \[
    \forall i \leq k \leq j - 1, \sigma(A_k) = A_{k + 1}, \sigma(A_j) = A_i, n_k = n_{k + 1}.
    \]

    Abusing notation, we denote the restriction to the product of these algebras by $\sigma_i: \prod_{k = i}^jA_k \rightarrow \prod_{k = i}^jA_k$. In this case, if $(X_k)_{k = i}^j$ is a fixed point then $\sigma(X_k) = X_{k + 1}$, i.e. this point is determined by its first component. Therefore,
    \[
    \dim A^{\sigma_i} = n_i^2 = \frac{1}{|j - i + 1|}\dim\big(\prod_{k = i}^jA_k\big) \leq \frac{1}{2}\dim\big(\prod_{k = i}^jA_k\big)
    \]
\end{itemize}

All in all, co-dimension of $A^\sigma$ can only come from those with $n_i \geq 2$, and from both inequalities above, we see that there is a lower bound for the codimension:
\[
L \geq \sum_{n_i \geq 2}\frac{n_i(n_i - 1)}{2} \geq \frac{1}{4}\sum_{n_i \geq 2}n_i^2.
\]
Since $\dim A = \sum_{i}n_i^2$, so we have
\[
\sum_{n_i: n_i  = 1}1 = \dim A - \sum_{n_i \geq 2}n_i^2 \geq \dim A - 4L = (1 - 4\epsilon)\dim A.
\]
\end{proof}

Because of Lemma \ref{matrix_anti_inv_lemma}, we can also have an estimate of the number of high-rank matrix algebras using the dimension of fixed points of an anti-involution.

\begin{corollary}
Suppose $A$ is a semi-simple algebra over $\mathbb{C}$, and $\sigma: A \rightarrow A$ is an anti-involution with fixed points $A^\sigma$ dimension $L = \epsilon\dim A (\geq \frac{1}{4}\dim A)$, then $A \cong \prod M_{n_i}(\mathbb{C})$, and with total dimension of at least rank $k (> 2)$ matrix algebras $\sum_{n_i: n_i  \geq k}n_i^2 \leq \frac{\epsilon - \frac{1}{4}}{\frac{1}{4} - \frac{1}{2k}}\dim A$.
\end{corollary}

\begin{proof}
Similar to the proof before, $\sigma$ factors to each component of matrix algebras with possibly a rearrange of indices,
\[
A \cong \prod A_i, \quad A_i \cong M_{n_i}(D_i) = M_{n_i}(\mathbb{C}).
\]

\begin{itemize}
\item Either $\sigma(A_i) = A_i$, and the restriction $\sigma_i: A_i \rightarrow A_i$ is again an anti-involution. Then
    \[
    \frac{n_i(n_i + 1)}{2} \geq \dim A^{\sigma_i} \geq \frac{n_i(n_i - 1)}{2}.
    \]

\item or there is a switch of two matrix algebras of the same dimension $\sigma(A_i) = A_j$, with $n_i = n_j$. Abusing notation, we denote the restriction to the product of these two algebras by $\sigma_i: A_i \times A_j \rightarrow A_i \times A_j$. In this case, $(X, Y)$ is a fixed point iff $\sigma(X) = Y$, i.e. this point is determined by its first component. Therefore,
    \[
    \dim A^{\sigma_i} = n_i^2 = \frac{1}{2}\dim (A_i \times A_j)
    \]
\end{itemize}

Now we look at matrices of different ranks:
\begin{enumerate}
\item For those matrix algebras with $n_i \geq k$, and from both inequalities above, we see that:
    \[
    \frac{\dim A^{\sigma_i}}{\dim A_i} \geq \frac{n_i - 1}{2n_i} \geq \frac{1}{2} - \frac{1}{2k}.
    \]
\item For those matrix algebras with $n_i < k$, and from inequalities we used in the proof of previous lemma, we see that:
    \[
    \frac{\dim A^{\sigma_i}}{\dim A_i} \geq \frac{n_i - 1}{2n_i} \geq \frac{1}{4}.
    \]
\end{enumerate}

Combining these two cases, we see that:

\begin{align*}
L = \epsilon\dim A &= \sum_{n_i \geq k}\dim A_i\frac{\dim A^{\sigma_i}}{\dim A_i} + \sum_{n_i < k}\dim A_i\frac{\dim A^{\sigma_i}}{\dim A_i} \\ &\geq \sum_{n_i \geq k}\dim A_i(\frac{1}{2} - \frac{1}{2k}) + \frac{1}{4}\sum_{n_i < k}\dim A_i \\ &= \sum_{n_i \geq k}\dim A_i(\frac{1}{2} - \frac{1}{2k}) + \frac{1}{4}(\dim A - \sum_{n_i \geq k}\dim A_i) \\
\Longrightarrow (\epsilon - \frac{1}{4})\dim A &\geq \sum_{n_i \geq k}\dim A_i(\frac{1}{4} - \frac{1}{2k}) \\
\Longrightarrow \sum_{n_i \geq k}\dim A_i &\leq \frac{\epsilon - \frac{1}{4}}{\frac{1}{4} - \frac{1}{2k}}\dim A.
\end{align*}
\end{proof}

\section{Proof of Main Theorem}\label{ch_proof}
Now that we have developed all the tools we need, let us prove our main results.

\begin{proof}[Proof of \ref{dim_thm}]
First we observe that
\[
(H \times H)_g \cong H_x.
\]

By Corollary \ref{closed_orb}, for any $g \in G(F)$ such that $HgH$ a closed orbit in $G$, $H(F)gH(F)$ is fixed by $\sigma$. Next, we count the number of such $H(F)gH(F)$ orbits. Under the symmetrization map, closed $H \times H$-orbits correspond to closed $H$-orbits in $G^\sigma$. Moveover, the symmetrization map induced an injection which by Proposition \ref{closesd_to_ss} maps closed orbits to closed orbits:
\[
G/H \xhookrightarrow{s} G^\sigma.
\]

Also, by Proposition \ref{closesd_to_ss} and Theorem \ref{ss_to_closed}, we see that closed $H$-orbits correspond to semi-simples $x \in G^\sigma$. From Corollary \ref{open_subset_ss}, we see that there is a subset $V \in G^\sigma$ of semi-simple elements. Suppose we start with an element $g \in U$ the open dense subset of $G$ in Corollary \ref{Zariski_open_subset}, then $s(g) \in V$ is semi-simple, and so $\mathrm{Ad}(H)s(g)$ is a closed orbit. Therefore, $HgH$ is a closed orbit.
\begin{figure}[htp]
\centering
\tikzset{node distance = 4cm, auto}
\begin{tikzpicture}
    \node (A) {$H \times H$-orbits in $G$};
    \node (B) [above of=A] {$g \in U$};
    \node (C) [right of=A] {$H$-orbits in $G^\sigma$};
    \node (D) [above of=C] {$s(g) \in V$ s.s.};
    \draw[right hook-latex] (A) -- (C);
    \draw[->] (B) to node [swap]{$HgH$}(A);
    \draw[->] (D) to node {$\mathrm{Ad}(H)s(g)$} (C);
    \draw[->] (B) to node {$s$} (D);
\end{tikzpicture}
\end{figure}

Now we need to consider similar situation over $F$. Suppose $g \in U(F)$ which is an open dense subset of $G(F)$, then $H(F)gH(F)$ is fixed by $\sigma$. We can have a refined result as in Theorem \ref{finite_to_one} to deal with orbits in $U(F)^c$.

Firstly, we claim that for any $g \in U(F)$ the fibre of $\pi_{G(F)}$ contains only one orbit, which is the unique closed orbit $H(F)gH(F)$. Suppose there is another orbit $H(F)g'H(F)$ in the fibre, then by the definition of categorical quotient, $\pi_{G(F)}(g) = \pi_{G(F)}(g')$ if and only if the closures of the orbits meet. However, $H(F)g'H(F) \subset U(F)^c$ which is a closed set, while $H(F)gH(F) \subset U(F)$, so their orbit closure cannot meet.

Secondly, we claim that same result of Theorem \ref{finite_to_one} happens over $F$, i.e. every fibre of the quotient map $\pi_{G(F)}$ contains finitely many $H(F) \times H(F)$-orbits. Suppose $\pi_{G(F)}(g) = \pi_{G(F)}(g')$ where $H(F)gH(F)$ is closed, then $\pi_G(g) = \pi_G(g')$ over $\bar F$, i.e.
\[
g \in \overline{H(F)g'H(F)} \subset \overline{Hg'H}.
\]

So each $H(F)\times H(F)$-orbits in the fibre of $\pi_{G(F)}$ is contained in an $H\times H$-orbits in the fibre of $\pi_{G}$, which we know there are finitely many. It suffices to show that there are finitely many $H(F)\times H(F)$-orbits contained in $Hg'H$ for $g' \in G(F)$. We have already come very close to this statement in the proof of Corollary \ref{closed_orb}, where the first Galois cohomology classifies $A$-torsors.

Now we fix $g \in G(F)$. If $g' \in G(F)$, we can define
\[
A \coloneqq (H \times H)_g, \quad Y' \coloneqq \{(h_1, h_2) \in H \times H \mid h_1g'h_2^{-1} = g \}.
\]

Then $Y'$ is an $A$-torsor (an $A$-set isomorphic to $A$), where $A$ acts on $Y$ by standard left multiplication. $Y'$ is defined as an $A$-torsor through a fixed point $(h_1, h_2) \in Y'$, and a map
\[
A \rightarrow Y, \quad (a_1, a_2) \mapsto (a_1h_1, a_2h_2).
\]

Recall we can also define a $1$-cocycle
\[
f': \mathrm{Gal}(\bar F / F) \rightarrow A, \quad \varphi \mapsto \big(\varphi(h_1)h_1^{-1}, h_2\varphi(h_2^{-1})\big).
\]

Now $H^1(F, A)$ is the set of equivalence classes of crossed homomorphisms, which corresponds bijectively to the isomorphism class of $A$-torsors. $H^1(F, A)$ is bounded by the number of irreducible components of $A$ which is finite. Suppose $g'' \in G(F)$ is another element, then $Y''$ is another $A$-torsor, with $f''$ defined by $(h_3, h_4) \in Y'$ another $1$-cocycle. $Y'$ and $Y''$ are isomorphic if their corresponding cocycle $f'$ and $f''$ is differed by a cochain, i.e. an element in $H(F) \times H(F)$. This means that $g'$ and $g''$ are in the same $H(F) \times H(F)$-orbit:
\[
g' = h_1^{-1}h_3g''h_4^{-1}h_2, \quad (h_1^{-1}h_3, h_2^{-1}h_4) \in H(F) \times H(F).
\]

Therefore, there are finitely many $H(F) \times H(F)$-orbits in $Hg'H$. From the proof, we also see that the upper bound of number of orbits in a fibre of $\pi_{G(F)}$ depends on $H^1(F, A)$ and $\dim G$, but not on the characteristic of $F$. Denote this upper bound by $C$.

Now we can compare the size of $Z^\sigma$ and $Z$. $\pi_G$ is a closed map by definition, so as a quotient map it is also a dominant map. Then $\pi_G(U)$ is a dense subset of $Y \coloneqq G \sslash H \times H$, and
\[
\dim\pi_G(U) = \dim Y, \quad \dim \pi_G(U^c) < \dim Y.
\]

Over $F$, we have
\[
Z = G(F)/H(F) \times H(F) = U(F)/H(F) \times H(F) \sqcup \pi_{G(F)}^{-1}\big(\pi_G(U^c)(F)\big)/H(F) \times H(F).
\]

On the one hand,
\[
\big|\pi_{G(F)}^{-1}\big(\pi_G(U^c)(F)\big)/H(F) \times H(F)\big| \leq C |\pi_G(U^c)(F)| \leq C|F|^{\dim Y - 1},
\]

since each fibre contains at most $C$ many orbits. On the other hand, Lang-Weil's bound tells use that
\[
|U(F)/H(F) \times H(F)| \geq \frac{1}{2}|F|^{\dim Y},
\]

which consists of points in $Z^\sigma$. Hence,
\[
\frac{|Z^\sigma|}{|Z|} \geq \frac{\frac{1}{2}|F|^{\dim Y}}{\frac{1}{2}|F|^{\dim Y} + C|F|^{\dim Y - 1}} = 1 - \frac{2C}{q + 2C} \geq 1 - \frac{C'}{q}.
\]
\end{proof}

\begin{proof}[Proof of \ref{e_Gel}]
To understand the permutation representation of $G$ on the cosets of $H$, a stand method is to look at the Hecke algebra $\mathcal{H} = \mathrm{End}_G\big(\mathbb{C}[G/H]\big)$. The permutation representation of $G$ on the cosets of $H$ is the induction of trivial representation on $H$, so by Frobenius Reciprocity, we have
\begin{align*}
\mathcal{H} &= \mathrm{Hom}_G\big(\mathbb{C}[G/H], \mathbb{C}[G/H]\big) \\
&= \mathrm{Hom}_G\big(\mathrm{Ind}_H^G\mathbb{C}_H, \mathrm{Ind}_H^G\mathbb{C}_H\big) \\
&= \mathrm{Hom}_H\big(\mathbb{C}_H, \mathrm{Res}_H^G\mathrm{Ind}_H^G\mathbb{C}_H\big) \\
&= \mathbb{C}[H \backslash G/H] = \mathbb{C}[Z].
\end{align*}

Now $\mathcal{H}$ is a semi-simple algebra, and we have the induced anti-involution on $\mathbb{C}[Z]$ from the anti-involution defined on $Z$, i.e.
\[
\sigma\big(H(F)gH(F)\big) \coloneqq H(F)\sigma(g)H(F).
\]
If $\dim\mathbb{C}[Z]^\sigma \geq (1 - \epsilon)\dim\mathbb{C}[Z]$, then we can apply Lemma \ref{multi_one_lemma} to get that $\mathcal{H} \cong \prod M_{n_i}(\mathbb{C})$, and the majority of the matrix algebras are of rank $1$:
\[
\big|\{n_i \mid n_i  = 1\}\big| \geq (1 - 4\epsilon)\big|\{n_i \mid n_i > 0\}\big|.
\]

Because of Schur's lemma, the decomposition of $\mathcal{H}$ is according to the decomposition of $\mathbb{C}[G/H]$ into irreducible representations. Each matrix algebra corresponds to an irreducible representation appeared in $\mathbb{C}[G/H]$ , and its rank $n_i$ is the multiplicity of this irreducible representation in $\mathbb{C}[G/H]$. Therefore, there are at least $(1 - 4\epsilon)$ portion multiplicity-one irreducible representations in $\mathbb{C}[G/H]$.
\end{proof}

\begin{proof}[Proof of \ref{multi_one_thm}]
Apply Theorem \ref{dim_thm} to Theorem \ref{e_Gel} with $\epsilon = \frac{C}{2q}$. Notice that when we pass from $Z$ to $\mathbb{C}[Z]$, since $\sigma$ is an anti-involution, the co-dimension of fixed points is half, so
\[
\frac{|Z^\sigma|}{|Z|} \geq 1 - \frac{C}{q} \Longrightarrow \frac{\dim\mathbb{C}[Z]^\sigma}{\dim\mathbb{C}[Z]} \geq 1 - \frac{C}{2q}
\]
\end{proof}

\end{document}